\documentclass[12pt, leqno]{amsart}
\usepackage{latexsym}
\usepackage{amssymb}
\usepackage{amsmath}
\usepackage[english]{babel}
\usepackage[colorlinks=true, pdfstartview=FitV, linkcolor=red, citecolor=red, urlcolor=blue, backref=page]{hyperref} 

\usepackage[top=1.25in, bottom=1.25in, left=1.25in, right=1.25in]{geometry}

\newtheorem{theorem}{Theorem}

\newtheorem{definition}[theorem]{Definition}

\newtheorem{lemma}[theorem]{Lemma}

\newtheorem{proposition}[theorem]{Proposition}
\newtheorem{remark}[theorem]{Remark}

\newcommand{\rn}[1]{\mathbb{R}^{#1}}

\newcommand{\re}{ \mathbb{R}}

\newcommand{\beq}{\begin{equation}}
\newcommand{\bea}[1]{\begin{array}{#1} }
\newcommand{\eeq}{ \end{equation}}
\newcommand{\ea}{ \end{array}}
\newcommand{\ep}{\epsilon}

\newcommand{\al}{\alpha}
\newcommand{\ga}{\gamma}

\newcommand{\de}{\delta}
\newcommand{\ds}{\displaystyle}
\newcommand{\ts}{\textstyle}
\newcommand{\rar}{\mbox{$\rightarrow$}}

\newcommand{\ran}{\rangle}
\newcommand{\lan}{\langle}
\newcommand{\Ga}{\Gamma}
\newcommand{\la}{\lambda}
\newcommand{\La}{\Lambda}
\newcommand{\ar}{\partial}
\newcommand{\si}{\sigma}
\newcommand{\Si}{\Sigma}
\newcommand{\om}{\omega}
\newcommand{\Om}{\Omega}
\newcommand{\be}{\beta}

\newcommand{\ph}{\phi}
\newcommand{\he}{\theta}

\newcommand{\Ph}{\Phi}
\newcommand{\hs}[1]{\mbox{$ \hspace{#1}$}}

\newcommand{\sem}{\setminus}

\newcommand{\ti}{\tilde}
\newcommand{\noi}{\noindent}

\renewcommand*{\backref}[1]{}
\renewcommand*{\backrefalt}[4]{%
 \ifcase #1 (Not cited.)%
   \or        (Cited on page~#2.)%
    \else      (Cited on pages~#2.)%
    \fi}  

\begin{document} 

\title[ On  $ d $ and $M$   Problems for Newtonian Potentials in Euclidean $ n $ space ] {On  $ d $ and $M$   Problems for Newtonian Potentials in Euclidean $n $ space} 

 \dedicatory{Dedicated to the memory of  Albert Baernstein II,  Ronald Gariepy, and  Walter Hayman}


\begin{abstract}
In this  paper  we first   make and discuss  a  conjecture concerning    Newtonian potentials in Euclidean n  space   which have all their mass on the  unit sphere about  the origin,  and are normalized to be one at the origin.  The conjecture essentially  divides these potentials  into subclasses   whose criteria for membership  is that a given member have  its maximum on the closed unit ball at most  M and its  minimum at least d.  It then  lists  the  extremal  potential in each subclass  which is  conjectured to   solve   certain  extremal  problems.    In  Theorem 1.1 we show existence of  these extremal  potentials.     In  Theorem 1.2   we prove  an   integral inequality  on spheres about the  origin,  involving  so called extremal potentials, which lends credence to the conjecture.     
    
\end{abstract}


\author[J. Lewis]{John Lewis}
\address{{\bf John Lewis}\\ Department of Mathematics, University of Kemtucky\\ Lexington, Kentucky} 
\email{johnl@uky.edu}

\maketitle
\setcounter{tocdepth}{2}
\tableofcontents

\section{Introduction}  
\label{sec1}    
Let $ n \geq  2 $ be a   positive integer,   $  x = ( x_1,  x_2, \dots, x_n ) $  a point in  Euclidean  $n$  space, $ \rn{n},  $  and let $ |x| $ denote the norm  of  $ x. $    Put  $ B ( x, r ) =  \{ x : | x | < r \} $  when $ r > 0. $   For  fixed  $ n \geq  2,  $  let  $  \mu $  be  a  positive  Borel  measure on 
$  \mathbb{S}^{n-1} = \{ x : |x| = 1\} $   with  $  \mu (\mathbb{S}^{n-1})   = 1. $    
   Next   let   
$  \mathcal{H}^{n-1} $ denote Hausdorff  $ n - 1 $ measure   and let   $    \Ph $ be  a non decreasing  convex function on  $  \re .  $  Finally  let      $ \mathcal{F} $ denote the family of  potentials  $ p $    satisfying      
\beq \bea{l}  \label{1.0}  (a)    \hs{.2in}     p ( x )  =  {\ds  \int_{\mathbb{S}^{n-1}}}  | x - y |^{2-n}  d \mu ( y ),   x \in \rn{n},   \mbox{ when $ n > 2, $}     \\ \\  (b)  \hs{.2in}  p ( x )  =  2 {\ds  \int_{\mathbb{S}^{n-1}} } \log \frac{1}{| x -  y|}  \, \, d \mu ( y ) ,  x \in    \rn{2}.   \ea \eeq    
  
  Theorems 1.1 and 1.2 in this paper are based on my efforts to prove the following conjecture: \\ 
  \\ 
{ \bf Conjecture 1 :} {\em  If   $ n \geq 3, $ there is  a  1 - 1  map  from     \\ \[  \{ ( \xi_1, \xi_2 ) : 0 \leq \xi_1 < \xi_2 
\leq \pi  \}  \rar  \{ (d,M) :  2^{2-n} \leq d  <  1,  1 < M \leq \infty \}   \]   for which there exists 
a potential $ P ( \cdot, d, M ) \in  \mathcal{F}, $   satisfying      
\beq \label{1.1}   \hs{.1in}   \bea{l} (a)  \hs{.2in}  \mbox{   $  P  ( \cdot, d, M)   \equiv  M $ on 
 $ E_1 =  \{ x \in  \mathbb{S}^{n-1} :   \cos \xi_1   \leq x_1  \leq 1 \},  $ }    \\   (b) \hs{.2in}     \mbox{  $ P ( \cdot, d, M  ) \equiv  d $  on 
   $ E_2 =  \{ x  \in  \mathbb{S}^{n-1} :  -1 \leq x_1  \leq \cos \xi_2   \}, $  }  
 \\  (c)  \hs{.2in} \, \,  P ( \cdot, d, M ) \geq  d  \in \bar B(0,1), 
\mbox{ and } P(\cdot, d, M) \leq M \mbox{ in } \rn{n},  \\ (d)  \hs{.2in}  \,  P ( 0, d, M ) = 1    \mbox{ and  $  P ( \cdot, d, M ) $  is harmonic   in  }   
\rn{n}  \sem   ( E_1 \cup E_2 ).   \ea \eeq   
  Given     $
 d = d (\xi_1, \xi_2),  M = M (\xi_1, \xi_2),   $ 
let  \[   \mathcal{F}_d^M   
       = \{ p \in \mathcal{F}  \mbox{ with $ d \leq p$  in $ \bar B ( 0, 1) $ and $ p \leq M $ in $ \rn{n}$ \} . } \]      If      $ 0  < r  < \infty,  $ and   $  p \in \mathcal{F}_d^M ,  $ then  
\beq \label{1.2}    \int_{\mathbb{S}^{n-1}}  \Ph ( p (r y ) )  d  \mathcal{H}^{n-1} y   
\leq    \int_{\mathbb{S}^{n-1}}  \Ph ( P ( r y, d, M ) )  d  \mathcal{H}^{n-1} y  \, . \eeq      } 

The analog of Conjecture 1 is true in $ \rn{2}. $  To briefly outline its proof we use complex notation.  So $  i = \sqrt{-1}, z =  x + iy, \bar z  =  x - iy,  e^{i\he} = \cos \he + i \sin \he, $ and  \\
$ B ( z_0, \rho)  = \{ z : | z - z_0 | < \rho  \}. $   Let  $  U $   denote the class of   univalent  (i.e, 1-1 and analytic)    functions   $ f $  satisfying  $ f ( 0 ) = 0,  f' (0)  = 1,  $  and  for which  $ D  = f ( B (0, 1 ) ) $ is starlike with respect to  0 (so each line segment connecting 0 to a point in $ D $ is  also contained in $  D). $  If $ f \in U, $ then 
    using  the  fact that  $ f ( B (0, r ) ) $  is  also starlike,  with respect to 0,  one can show    Arg $ f (r e^{i\he}) $  (i.e the principal argument of  $f$)  on $ \ar B ( 0, r) $  is non decreasing as  a  function of  $ \he $  so    if $  z = r e^{i\he}  \in  B (0, 1), $   then      \beq    
       \label{1.3}  - i  \frac{d}{d\he}  \, \log f (r e^{i\he} )  =   \frac{ d}{ d \he}  \, [\mbox{Arg } f ( re^{i\he} ) - i 
      \log |f(re^{i\he}|\,  ] =   z   f' (z)/ f ( z)   \eeq  and thus  Re $ ( z  f' (z)/ f ( z))    \geq  0, $  when $ z \in  B (0, 1). $              
  
      From \eqref{1.3} and the  Poisson integral formula  for   $ B (0, 1) $  it follows (see [D]) that   
           \beq \label{1.4}  \frac{ z f' (z)}{ f (z)}  =   \int_{\mathbb{S}^1} \frac{ 1 + e^{-i \he} z }{ 1 - e^{- i \he} z }  \,  d \nu ( e^{i\he} ),  \, \, \,     z \in B (0, 1),  \eeq    where $  \nu  $ is  a  positive  Borel measure on  
           $ \mathbb{S}^1 $  with  $  \nu ( \mathbb{S}^1 ) = 1. $  Note that if  $ f $ is sufficiently smooth on $  \bar B  ( 0, 1 ), $  then 
 
  \beq  \label{1.5}  \frac{d}{d\he}  \mbox{Arg } f (e^{i\he} ) = 2\pi  d \nu ( e^{i\he} )/d \he. \eeq  \eqref{1.5}   implies that  if  $ I   \subset \ar B (0, 1 )  $  is  an arc  and  $ f (I ) $  lies on  a  ray through 0,  then   $  \nu (I) \equiv 0. $  
             Dividing  \eqref{1.3},  \eqref{1.4}  by   $ z $  and  integrating  we get  
             \beq \label{1.6}  \log (f(z)/z)  =   -  2  \int_{ \mathbb{ S}^1} \log ( 1 - e^{-i \he} z ) \, d \nu (e^{i\he} )  \eeq  where $ \log $ is the principal  logarithm.      From \eqref{1.6} we see that   
                       if   $  \, \nu  = \mu, $ then    
  $ p (z)  =  \log |f(z)/z|, z \in B (0, 1 ), $  in \eqref{1.0} $(b). $     Given $ \ti d, 1/4 \leq \ti d < 1, $   and $ 1 <  \ti M \leq \infty, $    let  $ U^{\ti M}_{\ti d} $  denote starlike univalent functions   $f$ in  $ U $  satisfying   $ \ti  M  \geq | f (z)/z| \geq   \ti d $ in  $ B (0, 1 ). $           
Thus if  $  d = \log \ti d $ and  $ M = \log \ti M, $ then  for $ n = 2, $   \[  
  \mathcal{F}^M_d = \{ p(z) = \log |f(z)/z|  : f \in U^{\ti M}_{\ti d} \}  \]
  Using this  fact  one sees that the analogue of  $ P  ( \cdot, d, M) $  in   Conjecture 1 for $ n = 2$ is    
            $ P ( z, d, M )  =    \log | G ( z, \ti d, \ti M )/z |, z \in B( 0, 1 ), $  where  $ G ( \cdot, \ti d, \ti M )  $ maps   $ B (0, 1  ) $  onto  
             $   D = D (\ti d, \ti M ).$    If  $  1/4  <   \ti d  <  1 <  \ti M  < \infty, $  then $ D $ is the bounded keyhole domain described as follows. For some  $ \tau =  \tau ( \ti d, \ti M ),   0   <  \tau   < \pi,   \ar D $ is the union of the arcs :  
            \[     \{ \ti M e^{ i \he} : - \tau  \leq \he \leq \tau  \},  \{ \ti d e^{i \he} :  \tau  \leq \he \leq  2 \pi - \tau \}, \]   and the  line segments,  $  [\ti de^{i\tau}, \ti Me^{i\tau}],   [\ti de^{-i\tau}, \ti Me^{-i\tau}] . $             
   
               In    \cite{BL1}  
 we showed   for fixed $ \ti d, \ti M,  $  that  $ L (z, \ti d. \ti M) =   
\log (G (z, \ti d, \ti M  )/z ) $  maps  $ B ( 0, 1 ) $ univalently  onto a  convex domain containing  0  and  if $ f \in   U_{\ti d}^{\ti M } , $  then $ l ( z ) = \log (f (z)/z) $ is subordinate   to     $ L ( z, \ti d, \ti M ) $.  That is,   $ L^{-1} \circ l $ maps $   B ( 0, 1) $  into  $  B (0, 1).$  This result  was  in fact  a  corollary of  a  much more general  subordination theorem for    Mocanu  convex  univalent functions that are bounded above and below in the unit disk.   Our proof  used  a contradiction type argument and  the Hadamard - Julia variational formulas to determine the solutions to a  certain class of  extremal problems. Runge's theorem then gave subordination in the given class of Mocanu convex functions.  
 Conjecture 1 for $ n = 2, $  follows from   the  relationship between  $ p, P, $   and  $ f,  G,  $ as  well as properties of subordination (see  for example \cite{BL2}).\\
      
  Let   $ e_i $  denote the  point in $ \rn{n}$  with 1  in the $i$ th position and   zeroes elsewhere.   Conjecture 1 is true  in   $ \rn{n}, n >2,  $   when $ \xi_2 = \pi,  0 \leq \xi_1 < \pi , $  
so  $ E_2  = \{  - e_1 \},  $ and  $ 1 < M \leq \infty $.    It was   proved in 
   \cite{GL} .   Our   proof  used  a maximum principle  for  the  celebrated Baernstein  *  function  (see          
   \cite{H2} or   \cite{BDL}),   defined as  follows:  
         Given  $  x  \in \rn{n} \sem \{0\}, $ introduce spherical coordinates, $ r, \he $   by   $ r  = |x|,  x_1 = r \cos \he,  $   $  0  \leq \he  \leq \pi. $  Let  $ l$ be  a locally  integrable   real valued function   on 
         $ A  =  \{ x \in \rn{n}  : r_1 < |x|   < r_2  \}, $  and  set           
         \beq  \label{1.7}  l^{*}  ( r, \he)  =  \sup_{\La } \int_{\La}  l (ry ) d \mathcal{H}^{n-1} y  \eeq  
  where the  supremum is taken over all Borel  measurable sets  $  \La  \subset \mathbb{S}^{n-1} $     with \[ \mathcal{H}^{n-1} ( \La ) =  \mathcal{H}^{n-1} ( \{y \in \mathbb{S}^{n-1} : y_1 \geq \cos \he \}) .\] 
One can show that  \eqref{1.2}  in Conjecture 1 is   equivalent   to    
   \beq  \label{1.8}  p^*  ( r, \he )  \leq   P^* (r, \he, d, M ) ,   \mbox{ whenever } 0 < r < \infty, \, 0  \leq  \he \leq \pi.   \eeq 
	  For  a proof of this  equivalence  see  section 9.2 in  \cite{H2}.    \\

For  $ n = 2, $    Professor   Baernstein  in    \cite{B} showed  that if  $ l $ is subharmonic   in $  A,  $  then $ l^* $  in   \eqref{1.7} is subharmonic in   
           $  I = \{  z = r e^{i\he}   :   0 <  \he < \pi,  r_1  < r  < r_2 \}. $     
 Moreover  if  \[ \bea{l}
  (a) \hs{.2in}   l ( r e^{i \he} )  =  l( r e^{- i\he}), \,  0 \leq \he \leq \pi,  \\ 
 (b) \hs{.2in}   l ( r e^{i\he} )  \mbox{   is  non increasing on $  [0, \pi],$    
  for fixed $ r,  r_1 < r < r_2, $  } \\  (c) \hs{.2in} \mbox{$ l$ is harmonic in  $A$, } \ea \]   then  $ l^* $ is harmonic  in  $ I . $   
 In $ \rn{n}, n > 2, $  $ l^* ( r, \he )  $   need not  be subharmonic  in   $ I$  even   when $ l $  is harmonic in 
 $ A.  $    Instead we  proved (see also \cite{BT}) the  following maximum principle: \,  Suppose  $ l $ is subharmonic in  $ A $  and   $ L $ is harmonic in  \[   \Om =  \bigcup_{r_1 < r  < r_2}    \{ x :  |x| = r,  x_1 >  \cos ( \he (r) )\},      0 < \he (r) \leq  \pi,  r  \in  (r_1, r_2),  \]   with $ L = L ( r,  \he ) $  symmetric about the $ x_1 $ axis and $ L(r, \cdot )$ non increasing on \\  $ [0, \he (r) ) ,  r_1 < r < r_2. $  
Then  \[    (  l^* -  L^* ) (x)    <  \sup_{\Om}    ( l^* -  L^*) , \mbox{ whenever }  x \in \Om, \mbox{ unless } L \equiv l.  \]  

 To briefly  outline  the   proof of   Conjecture 1 in \cite{GL},  for  given   $ \xi_2 = \pi, $\\ $  0 < \xi_1 < \pi,  $  existence of  $P= P ( \cdot, d, M ) ,  $  satisfying \eqref{1.1},   can be deduced from  a working knowledge of such tools for harmonic functions as  (a)  Wiener's criteria for solutions to the Dirichlet problem, 
 (b) the maximum principle for harmonic functions,   (c)  the Riesz representation formula for superharmonic functions, and  (d) invariance of the Laplacian under  reflection about planes containing  zero (see the proof of  \\ Theorem 1.1 for more  elaborate details).   
  Let  $  \Om =  \{ x :  - x \in  \rn{n} \sem ( E_1  \cup   ( 0,   \infty] ) \},$     $    L (x) =-  P ( - x, d, M ), $   \mbox{ and } $ l (x) 
= - p ( - x ), $  for $ x \in \rn{n}. $      
If   $ {\ds   \sup_{x \in \Om}  ( l^*   -  L^*) (x)  > 0,} $    then from  subharmonicity of  $ - p,$  harmonicity of $ P $  in $ \Om, $   decay of both potentials at $ \infty, $  and  the above maximum  principle, it  follows that  
there exists   $ y  \in   \ar  \Om  \cap  \mathbb{S}^{n-1} $  with  spherical coordinates  $ |y| =  1,  \hat \he,$   $ \pi  - \xi_1  \leq  \hat \he  \leq \pi, $  and  $    l^* ( 1, \hat \he ) - L^* ( 1,\hat  \he )   >  0. $ 
Since   $  l \geq - M $  and  $ L = - M $  on  $  - E_1,   $   we then obtain    
\[   0 <    ( l^*   -  L^*)  ( 1, \hat \he)  \leq   l^* (1, \pi ) -  L^* (1, \pi )  = 0.  \]  From this contradiction we conclude  that  $ l^* \leq L^* $ in $ \rn{n}. $ 
Next using  $ l^* ( r, \pi )  =  L^*(r, \pi ), 0 \leq r < \infty,  $  and  whenever $q \in \{ l, L\},$  that \[ (-q)^* (r,   \he )  =   q^* ( r,   \pi -  \he )  -    q^* ( r, \pi ) , 0 < r < \infty,  0  < \he  \leq  \pi ,  \]   we get  $  p^*  \leq P^*  $  in  $ \rn{n}, $ which  as mentioned in   \eqref{1.8}  implies   \eqref{1.2}.   

We have not been able to prove  \eqref{1.2} in Conjecture 1  for any other values of $ \xi_1,  \xi_2. $  However in  \cite{L}   we  used a mass  moving  method in  \cite{S} to show that  if   $ \xi_1 = 0, 0 < \xi_2 < \pi, $  and $ p \in \mathcal{F}_d^{\infty},  $    then  
    for  $ 0 < r  \leq 1, $   \beq  \label{1.9}   \max_{x \in B(0,r)} p ( x )   \leq  P(r, 0, d, \infty )   \mbox{ and }   \min_{x \in B(0,r)} p ( x )   \geq P( r, \pi, d, \infty) . \eeq  \\
  Moreover in this paper we prove the first part of our conjecture:    
 \begin{theorem}    \label{thm 1}   If  $ n \geq 3, $ there is a  1-1 map  from 
\[  \{ (\xi_1, \xi_2 ) :  0 \leq \xi_1 <  \xi_2 \leq \pi \}  \rar  \{ (d, M ) : 2^{2-n} \leq d  < 1 < M \leq \infty \}, \] for which there exists  a  potential   $ P = P ( \cdot, d, M ) 
\in \mathcal{F}  $  satisfying  
\eqref{1.1}. \end{theorem} 
Also we prove,   
 \begin{theorem}    \label{thm 2}   Given $  P ( \cdot,  d,  M )    $   as in  Theorem 1.1.    If   $ 0  < \xi_1  < \pi,  \xi_2 = \pi,  $  and $ p \in  \mathcal{F}_d^M , $   then     
\beq \label{1.10}   \int_{\mathbb{S}^{n-1}}  \Ph ( p (r y ) )  d  \mathcal{H}^{n-1} y      
 \leq  \int_{\mathbb{S}^{n-1}}  \Ph ( P ( r y, d, \infty  ) )  d  \mathcal{H}^{n-1} y.  \eeq     \end{theorem}     
  As for the plan of this paper,  in  section  2  we set the stage for the proof of Theorems 1.1 and  1.2  by  stating and/or proving several definitions and  lemmas.  In section 3  we prove  Theorem 1.1.   In section  4  we prove   Proposition 4.1, a rather tedious  calculation of mixed partials for a certain function.   In section  5,   Proposition 4.1  is used  to  prove  Theorem 1.2.  After each theorem we make  remarks and queries.   
   
\section{Notation, Definitions, and  Basic  Lemmas} \label{sec2} 
\label{sec2}  
\setcounter{equation}{0} 
 \setcounter{theorem}{0}

 Throughout this paper  $ c (a_1, \dots, a_m ) $  denotes a positive constant $ \geq 1, $  depending only on  $ a_1, \dots, a_n, $  and $ n. $ Also    $  A  \approx B $   means     $    A/B $ is bounded above and below by positive constants whose dependence will be stated.    As in  section 1   let $ dx $ denote  Lebesgue measure on    $ \rn{n},   \bar F $
the closure of $ F, $  $ d ( x,  F ) $  the distance from  $ x  $  to  the  set $ F,  e_i $ the point in  $\rn{n} $  with 1 in the $i$ th position and zeroes elsewhere, $ \lan \cdot, \cdot \ran $ the inner product in $ \rn{n} $, 
$  B ( x, r ) = \{y \in \rn{n} :  |y - x | < r  \}, \,  \mathcal{H}^k =  $  Hausdorff  
  $k$ measure, in
$ \rn{n}, 0 <  k \leq n. $ 
\begin{definition}  If $ O $ is an open set in $ \rn{n}, n \geq 3,  $  and $  F  \subset O $  is compact,   then  the Newtonian capacity of  $ F, $   denoted $ C(F), $   
is defined to be   
\beq  \label{2.1}   C (F)  =   \inf  \int_{\rn{n}}  |\nabla \ph |^2  dx,  \eeq  where $ \nabla \ph $  denotes the gradient of $ \ph $  and   the infimum is taken over all 
$ \ph \in C_0^\infty (\rn{n}) $ with $ \ph \equiv 1 $ on $F.$  \end{definition} 

 \begin{remark}   Recall   that  a bounded open set  $ G \subset \rn{n} $  is said to be  a  Dirichlet domain if  for any  continuous  real valued function $ q $ defined on $ \ar G, $  there exists a  harmonic function $ Q  $ in   $ G $ with 
\[  \lim_{x\rar y }  \, Q (x) = q(y),  \mbox{ whenever }  y  \in \ar G. \]    {\em Wiener's  criteria}  for   a bounded open set   $  G $ to be a  Dirichlet domain  states:   If   
\beq \label{2.2}  \int_ {0}^1  r^{1-n}  C( B_r  ( y ) \cap \ar G ) dr  =  \infty  \mbox{ for all }   y \in \ar G,    \eeq   then $ G $   is a Dirichlet domain.  
\end{remark} 
  
Given $ \xi_1, \xi_2,  0 < \xi_1 < \xi_2 < \pi.  $ Let (as in  Theorem 1.1), 
\[  E_1 =  \{ x  \in \mathbb{S}^{n-1} :  x_1 \geq \cos \xi_1 \},  E_2 =  \{ x  \in \mathbb{S}^{n-1} :   x_1 \leq \cos \xi_2 \}. \]   Put   $   \Om ( \xi_1, \xi_2 )  
=  \rn{n} \sem ( E_1 \cup E_2 ). $   For $ i = 1,2,$ one can show (see   [H1] ) that  for $ i = 1, 2, $  \\
 \[   r^{n-2}   \leq  c ( \xi_1, \xi_2 ) \, C ( B (x, r ) \cap  E_i )   \mbox{ whenever }  x \in E_i \mbox{ and }    0 <  r  \leq   \min ( \xi_1, \pi - \xi_2 ). \]  Using  Wiener's criteria  and the boundary maximum principle for harmonic functions it   follows that given $ R > 2 $  there exists $ \om_{i, R} $   harmonic  in $ B (0, R ) \cap \Om ( \xi_1, \xi_2 ) $   with continuous boundary values $ \om_{i,R} = 1 $ on $ E_i $   and  $ \om_{i,R}  \equiv  0  $ on $ E_j, j \not = i. $    Using the boundary maximum principle for  harmonic functions we find  that if $ R_1 < R_2,$   then $ |  \om_{i ,R_1} -  \om_{i, R_2} | \leq R_1^{2-n}. $  From  this fact  we deduce that 
$ \lim_{R \rar \infty} \om_{i, R}  =  \om_i $ uniformly on compact subsets of $ \rn{n} $ and  $ \om_i, i = 1, 2, $  is  harmonic in $ \Om (\xi_1, \xi_2) $ with  continuous boundary values  1 on $ E_i $ and  0   on $ E_j,  j \not = i $. Also $ \om_i (x)  \rar 0   $ as $ |x| \rar \infty $   and  $ \om_i $   is superharmonic in  an open set containing    $ E_i, $  as well as,  subharmonic in an open set containing $ E_j,  j \not = i. $  From these observations and the Riesz  representation theorem for  sub -super harmonic functions (see  [H1]), it  follows that there exists finite positive Borel measures $ \nu_{i,j},  i, j = 1,2, $  with the support of $ \nu_{i,1} \subset E_i $ while the support of  $ \nu_{i,2} $ is contained in  $ E_j,  j \not = i.$   Moreover  
\beq \label{2.3}   \om_i (x)  =  \int_{E_i}  | x -  y |^{2-n}  d\nu_{i,1} (y)     -  \int_{E_ j}  |x-y|^{2-n} d \nu_{i,2} (y),   x \in \rn{n}, \eeq (once again $E_j \not = E_i$).

Next for $ i = 1, 2, $  let  
\[  \la_{i}  (\he_1) = (  \nu_{i,1 } + \nu_{i,2} )  ( \{ x \in \mathbb{S}^{n-1}  :  x_1 \geq \cos \he_1 \} )  \]   and note that $ \om_i $ has boundary values that are symmetric about the $ x_1 $ axis. Thus from the boundary maximum principle for harmonic functions and invariance of the Laplacian under rotations we have :   $ \om_i ( x )  =  \om_i ( r, \he ) $ for $ i = 1, 2,$  whenever  $ r = |x|, x_1 = r \cos \he. $ Using this fact and arguing as in  section 2 of [L] we get  the Lebesgue - Stieltjes integral : 
\beq \label{2.4}  \om_i (x)  =  \om_i  ( r, \he )  =   c_n  \int_0^\pi  h ( r, \he, \he_1)  d \la_i ( \he_1 ),   x \in  \rn{n}.    \eeq 
Here  \beq  \label{2.5}  h( r, \he, \he_1 )  = \int_0^{\pi}  ( 1 + r^2 - 2r \psi  ( \he, \he_1, \ph) )^{1 - n/2}  \sin^{n-3} \ph  \, d \ph  \eeq 
with  \beq \label{2.6} \psi( \he, \he_1, \ph )  = \cos \he \cos \he_1  + \cos \ph \sin \he \sin \he_1 \eeq 
and  $ c_n $ is chosen so that  $ c_n h( 0, \he, \he_1 ) \equiv 1. $    With this notation we prove, 
\begin{lemma} \label{lem2.3} Given $ r > 0,   \frac{\ar \om_1 (r, \he ) }{\ar \he}  < 0, $ and $  \frac{\ar \om_2 (r, \he ) }{\ar \he}  > 0 $ on   $ (0, \pi )  $  when  $   0 < r  < \infty,  
 r  \not = 1, $  and these inequalities also hold when  $ r = 1,  \he \in ( \xi_1, \xi_2 ). $   \end{lemma} 
This  lemma  is essentially  trivial in  $ \rn{2} $  but perhaps not so obvious in  $ \rn{n},  n > 2,$  so we give  some details.      
\begin{proof} To  begin the proof of Lemma \ref{lem2.3}   let      $\ti  \Si $ be  a    plane   containing the origin with  unit normal, $ {\bf n  },  $   satisfying      $   \lan    e_1,  {\bf n  } \ran  = -  \sin  \he  $  for some $ \he,   0   <  \he <  \pi . $    
 Let   
\[   H_1 = \{ x : \lan x, {\bf n} \ran < 0 \} \mbox{ and }   H_2 = \{ x : \lan x, {\bf n} \ran  >  0\} .\] Clearly $ H_1 $ contains  $ e_1$ 
and  $ H_2 $  contains $ - e_1 . $   If  $ x    \in   \bar H_1 $    let   $ \ti x \in  \bar H_2  $  be the reflection of $ x \mbox{ in } \ti \Si $ defined by  $  \ti x  =  x  - 2 \lan x, {\bf n} \ran  \, {\bf n} . $      We  claim  that  if  
 \beq   \bea{l}  \label{2.7} (a)   \hs{.2in} \ti x \in  \bar  H_2 \cap E_1   \mbox{  then }  x \in  \bar H_1 \cap E_1   
      \\ \mbox{ while if }     \\ 
              (b)   \hs{.2in}  x \in  \bar H_1 \cap E_2   \mbox{  then }  \ti x \in  \bar  H_2 \cap E_2 .  
      \ea \eeq       

 To prove  our  claim we assume, as we may, that   
	$  \lan  e_i,  {\bf n} \ran = 0,  3 \leq i \leq n, $   which is permissible  since $ E_1,   E_2 $   are symmetric about the   $ x_1 $  axis. 
Then $ { \bf n} = - \sin   \he \,  e_1 + \cos  \he\,  e_2  $  for some $ 0  < \he  <  \pi $  and  $  \hat e  = \cos \he \, e_1 +  \sin \he \, e_2 \in  \ti \Si.$  
     We note that  if     $ x \in H_1 \cap  \mathbb{S}^{n-1},$   then   $ x =  \al  {\bf n }   + \be  \hat e  +  \ga  e'   $  where $  {\bf n}, \hat e, e' $ are orthogonal unit vectors with $ \lan e' , e_1 \ran = 0 $   and  $ \{ \al, \be, \ga  \} $ are  any real numbers with  $ \al  < 0 $ and 
$  \al^2 +\be^2 + \ga^2 = 1.$   Also \\ $ \ti x = - \al  {\bf n }   + \be  \hat e  +  \ga  e'   $  and $ H_2 \cap  \mathbb{S}^{n-1} =  \{ \ti x  :  x \in   H_1 \cap  \mathbb{S}^{n-1} \}. $  Using this notation we see that   
$ \ti x_1   =    - |\al| \sin \he  + \be   \cos \he   \leq  x_1  = |\al| \sin \he  +  \be   \cos  \he $ with strict inequality unless $ \al = 0. $    Thus  \eqref{2.7} $(a)$ is true.   \eqref{2.7} $(b) $  is proved similarly so we omit the details.    To  finish the proof of  Lemma \ref{lem2.3}   we observe  from  \eqref{2.7} $ (a) , $  the continuous boundary values of $ \om_1, $ 
and Harnack's inequality  for positive harmonic functions  that either  $  \om_1 (  x )  -  \om_1 (\ti x) > 0,   x  \in H_1, $ 
or  $ \Om ( \xi_1, \xi_2 )  $ is symmetric  about  $  \ti \Si. $    However  this cannot happen if   $ \lan {\bf n},   e_1 \ran = - \sin  \he,  \he \in (0, \pi), $   as  
$  E_1 \cap \rn{2} $ is not even symmetric  about the line  through the origin and  $ \hat e . $   Similarly   $ \om_2 ( x ) -  \om_2 ( \ti x )  < 0 $  in  $ H_1. $   
  Lemma  \ref{lem2.3} now  follows from these inequalities and the  Hopf boundary maximum principle for harmonic functions (see \cite{E}).   \end{proof}

Finally in this section we  state 
\begin{lemma} \label{lem2.4}  ${ \ds \frac{\om_1}{ 1 - \om_2}  ( x ) }\rar \hat \ga,  0  <  \hat \ga  \leq 1,   \mbox{ $ as $ } x \in  \Om ( \xi_1, \xi_2 )  \rar  y  \in  E_2 $    \mbox{ with }        $  y_1  =  \cos \xi_2. $  \end{lemma}  \begin{proof} we note that  
\beq  \bea{l} \label{2.8}   1- \om_2 = \om_1 + \ti \om \mbox{ where $ \ti \om $ is  harmonic in $ \Om ( \xi_1, \xi_2 )$ with continuous boundary value 0} \\
\mbox{ on $  \ar \Om ( \xi_1, \xi_2) $ and $ \ti \om (x) \rar 1$  as  $ x \rar \infty. $} \ea \eeq        Lemma \ref{lem2.4} follows  from  \eqref{2.8}  and essentially a boundary Harnack inequality in \cite{KJ} even though  $  \Om( \xi_1. \xi_2) \cap B ( y, \rho),$ is not an NTA domain for any $ \rho > 0.$   To give  a  few details,  if  $ z \in \mathcal{S}^{n-1} $  
and $  \rho =  (y_1 - z_1 )/100 > 0,   $  then  $ \om_1/(1-\om_2) \approx  1  $  in  $ B ( z, \rho) \sem  E_2 $  is easily shown  using  barriers. Applying Harnack's inequality for positive harmonic functions, we then obtain   this inequality in  $ \Om ( \xi_1, \xi_2 ) \cap \{ x : |  x_1 - y_1|  \leq \rho/2 \}.  $  After  that one shows for some $ c = c ( \xi_1, \xi_2) > 1, $  
 \[ \mbox{ osc}_{B(y, r/2) }  \, \, \frac{ \om_1}{1-\om_2}  \leq  (1 - 1/c)  \mbox{ osc}_{B(y, r) }  \, \, \frac{ \om_1}{1-\om_2}     \mbox{ for } 0 < r \leq \rho/2 .  \]  
  Here $ \mbox{ osc}_{B(y, r/2) }  $  denotes oscillation on $ B ( y, r/2) \cap \Om (\xi_1, \xi_2). $    An iterative argument then gives H\"{o}lder continuity of 
 $ \frac{ \om_1}{1-\om_2} $  in  $ \Om(\xi_1, \xi_2) \cap B(y, \rho/2). $   
Thus  $ \hat \ga $ exists  and $ \hat \ga \leq 1 $ since $ \ti \om > 0 $ in  $ \Om ( \xi_1, \xi_2 ). $  
\end{proof} 

\section{Proof of Theorem 1.1}\label{sec3} 
\label{sec3}  
\setcounter{equation}{0} 
 \setcounter{theorem}{0}

\begin{proof}    
In  the  proof of  existence  for $ P ( \cdot,  d, M ) $  in  Theorem  1.1  we   assume that \\  $ 0 < \xi_1 <  \xi_2 <  \pi, $  as existence of  $ P ( \cdot,  d, M ) $  when $ E_1 = \{e_1\} $  was proved in [L] and if $ E_2 = \{-e_1\},$    
 $ P ( \cdot, d, M ) =  M  \om_1. $

 To begin the proof of this theorem  let          $    \ga  =  \inf_{ \he \in  ( \xi_1, \xi_2) }  \frac{ \om_1 (1, \he ) }{ 1 - \om_2 ( 1, \he ) } < 1. $  We assert that $ \ga = \hat \ga $ where $ \hat \ga $ is as in Lemma \ref{2.4}. 
Indeed   since $ \om_1 < 1 - \om_2  $ in  $ \Om ( \xi_1, \xi_2 ) $  and   $ \om_1, 1 - \om_2, $ have continuous boundary value $0$  when $ \xi_2 \leq \he \leq \pi, $   we see 
from  \eqref{2.8} that either 
\beq \label{3.1}  1 > \ga =   \frac{ \om_1 (1, \he_0 ) }{ (1 - \om_2) ( 1, \he_0 ) }  \mbox{ for some } \he_0 \in (\xi_1, \xi_2) \mbox{and/or }   \ga = \hat \ga . \eeq   To  show the first possibility cannot occur,  observe that it implies \\ $  \ga \om_2  ( 1, \he_0 )  + 
 \om_1 ( 1, \he_0 ) = \ga.$  Also,  $ \ga \om_1  + \om_2 $  is harmonic in  $ \Om ( \xi_1, \xi_2) $ with continuous boundary value $ \ga $ on $ E_2 $ and continuous boundary value $ 1  $ on $E_1.$  Using the fact that $ \ga $ is the minimum and  1 the maximum of $ \ga \om_2 + \om_1 $  on $ \mathbb{S}^{n-1} $   we can  essentially repeat the  proof of  Lemma \ref{lem2.3} to arrive at   
\beq \label{3.2}   \frac{\ar  (  \ga \om_2 + \om_1) (r, \he ) }{\ar \he}  < 0 \mbox{  when }  0 < \he < \pi, r \not = 0, 1,   \mbox{ and for } r=1,  \he \in  (\xi_1, \xi_2). \eeq
Thus  $  (\ga \om_2 +  \om_1)(1, \cdot )  $  is strictly decreasing on $ [\he_0, \xi_2 )$,  a contradiction  to   \eqref{3.1}.  

Let  $ V (x) =  a  ( \ga \om_2 + \om_1), $  where $ a > 0 $ is chosen so that  $ V (0) = 1. $   We shall show that  $  V =   P ( \cdot , d, M ) $  where $ M = a, d = a \ga$  to  complete the proof 
of  Theorem 1.1 except for showing the map is 1-1.  For this  purpose we note from  \eqref{2.3} that    
 \beq \label{3.3} V ( x )   =   \int_{\mathbb{S}^{n-1} }  |x-y|^{2 - n} d \si (y ) ,  x  \in \rn{n},  \eeq  where $ \si $ is  a  signed measure  with support in  $ \mathbb{S}^{n-1}$ and  of finite total variation.    Moreover   since $ V  \leq  M $  in $ \rn{n}$  with  $ V \equiv M $ on $ E_1, $   we see that $ V $ is superharmonic in an open set containing  $ E_1 $  and consequently (see  \cite{H1}),   
$ \si|_{E_1} $ is a positive Borel measure.  It remains to prove  $ \si|_{E_2} $ is a positive Borel measure in order  to conclude existence of $ P ( \cdot, d, M ) $ in Theorem 1.1.    
For this purpose let    $ \tau \in  \mathbb{S}^{n-1},$  with $  \tau_1 = \cos \he. $ 
Let $   x = r \tau,   0  < r < 1. $     Differentiating  \eqref{3.3}  we obtain  the Poisson integral (see \cite{H1} ) 
  \beq  \label{3.4}    r^{2 - n/2}  \,  \frac{ \ar  ( r^{n/2 - 1 } V ( r \tau  ) ) }{\ar r}  = (n/2 - 1 ) 
 \int_{\mathbb{S}^{n-1} }   \frac{ 1  - r^2}{ | r \tau -  y |^n}   d \si (y) .\,  \eeq    From properties  of the Poisson integral 
  (see [H1] ),  and  $ V (x) = V(r, \he), $   we deduce from \eqref{3.4} that   
\beq \label{3.5}  \lim_{r \rar 1}  \frac{ \ar  ( r^{n/2 - 1 } V ( r, \he   ) ) }{\ar r} =   
\de_n  (n/2 - 1 ) \frac{d \si}{d \mathcal{H}^{n-1}} = g (\he )   \eeq  for $ \mathcal{H}^{1} $ almost every $ \he 
\in  (0,\pi) $ where  $ \de^{-1}_n =  \mathcal{H}^{n-1}  ( \mathbb{S}^{n-1} ). $  
 Moreover,  
\beq \label{3.6}   g ( \he) = 0  =  (n/2 -1) V (1, \he) +   \frac{\ar V }{\ar r} ( 1, \he), 
\xi_1 < \he < \xi_2,   \eeq  since $ V $ is harmonic in $ \Om ( \xi_1, \xi_2 ). $    
Also since   $ \ar  \Om ( \xi_1, \xi_2) $ is  smooth at points $y \in  E_1 \cup E_2 $ with $ y \not =  \cos(\xi_1), \cos(\xi_2), $ it follows from Schauder type arguments (see [E]) that   
$ g $ is infinitely differentiable  on $ [0,\pi] \sem \{ \xi_1, \xi_2 \}. $  
From this observation and \eqref{3.2} we  deduce  that    \beq \label{3.7}   \frac{ d  -  V ( r,    \he_1  )  }{ 1 - r }  \leq   \frac{ d   - V ( r,  \he_2   ) }{ 1 - r } \,  \mbox{ for }    \xi_2 \leq \he_1 < \he_2 \leq   \pi \mbox{ and } r < 1.  \eeq     From \eqref{3.7},  \eqref{3.5},   and the mean value theorem from  calculus  we arrive at  
    \[      \liminf_{r \rar 1}  \left[ \frac{  d    -  V ( r, \xi_2 ) }{ 1 - r }\right]   \leq  -   d \, (n/2 - 1 )   +  g ( \he )  \mbox{ for } \he \in  (\xi_2, \pi].   \] 

 Thus to show  that  $ \si|_{E_2}  \geq 0, $  it suffices to show 
     \beq   \label{3.8}  \liminf_{r \rar 1}  \left[ \frac{  d    -  V ( r, \xi_2 ) }{ 1 - r }\right]   \geq  -   d \,  (n/2 - 1 ).  \eeq    

 To  prove   \eqref{3.8} we need some  boundary Harnack inequalities  in 
\cite{DS1} (see also \cite {DS2}). 
To set the stage for these inequalities  
put   $ y = (y_1, y' ) $  where  
\beq  \label{3.9}  y' = (y_2, \dots, y_{n}) \mbox{ and }   y = -  \frac{ x -  e_1 }{ | x - e_1 |^2 } - e_1/2   =  T(x) \mbox{ for }     x \in \rn{n}.   \eeq      
Then $ T $ maps  \beq  \label{3.10} \bea{l}  (a) \hs{.2in}  B(0,1)  \mbox{ onto }  \{y: y_1  >  0 \},  \\ \\    
   (b) \hs{.2in}  (\rn{n} \cup \infty )  \sem  \bar B(0, 1)    \mbox{ onto }  \{y: y_1  <   0 \},  \\ \\ 
  (c) \hs{.2in}     \mathbb{S}^{n-1}    \mbox{ onto }  \{y: y_1  =  0  \} \cup \infty. \\ \\
 (d) \hs{.2in}    E_1
  \mbox{  onto }   \{ y: y_1 = 0 , |y'|  \geq (1/2)  \cot (\xi_1/2) \} = F_1. \\ \\ 
  (e)  \hs{.2in}  E_2  \mbox{  onto }   
  \{ y: y_1 =  0 , |y'|  \leq  (1/2)  \cot  (\xi_2/2) \} =  F_2,     \ea \eeq 
  We note that    
    \[   x  =  T^{-1} (y)   =  e_1  -    \frac{y + e_1/2}{ |  y+e_1/2 |^2}  \mbox{ when }   y \in  \rn{n} \cup \infty .  \]   
 Using  this note, the Kelvin transformation (see \cite{H1}),  and translation invariance of  harmonic functions  we find that  if $ \hat u $   is   harmonic at $x,$  then   
      \[     \hat v ( y )   =   | y  + e_1/2 |^{2-n} \, \,  \hat u (  T^{-1} (y) ), \]  is harmonic at  $y \in \rn{n}$  (i.e,  in a neighborhood of $y$).   
    From this deduction we   conclude for fixed   $ \xi_1, \xi_2,  0 < \xi_1 <  \xi_2  < \pi, $  that if $ w_i (y) =  | y  + e_1/2 |^{2-n} \om_i (x),  i = 1, 2, $  then  
                \beq     \bea{l}  \label{3.11}  (a) \hs{.2in}   w_1, w_2 \mbox{ are continuous on  $ \rn{n} $ and harmonic in $ \rn{n} \sem (F_1 \cup F_2), $ }  \\  \\  (b) \hs{.2in}      w_1 \equiv 0   \mbox{ on $ F_2$ and } w_1 (y)  = | y + e_1/2|^{2 - n },    y \in  F_1, \\ \\  (c) \hs{.2in} 
    w_2  \equiv 0   \mbox{ on   $ F_1$ and }   w_2 (y)  =  |y + e_1/2|^{2-n},  \, y \in  F_2 ,
  \\  \\ (d) \hs{.2in}  w_1 (y), w_2 (y) \rar 0 \mbox{ as  $ y \rar \infty $, } \\ \\ 
 (e) \hs{.2in}  w_i ( \pm\,  e_1/2) = \om_i (0), i = 1, 2, \\  \\ (f)\hs{.2in} w_i ( y_1,  y'  ) = w_i (-  y_1, y' ), y \in \rn{n}, i = 1, 2,  \\ \\  (g) \hs{.2in}  w_i (y) = w_i ( y_1,  |y'| )  = w_i (y_1, \rho) , y \in \rn{n},  i = 1, 2 .   \ea \eeq             
  From \eqref{3.9}   we see   that  if  $ r = |x|, x_1 = r \cos \xi_2, $ then  
\beq   \bea{l}  \label{3.12}  (a) \hs{.2in}    y_1    =   {\ds \frac{(1/2) (1 -  r^2)}{ 1 + r^2  -  2r \cos \xi_2 }  \mbox{ and }   \rho  = |y'|  =  \frac{r \sin  \xi_2 }{ 1  + r^2 - 2 r \cos \xi_2   } , } \\ \\
(b) \hs{.2in}  1 - r^2  = {\ds  \frac{2 y_1}{ |y + e_1/2|^2} }  .    \ea \eeq     
 	Let $  w_3 (y)   =   | y + e_1/2|^{2-n} -   w_2 ( y ),  y  \in \rn{n}.  $   Then  $ w_3 = w_3 (  y_1, \rho) $ is harmonic in  $ \rn{n} \sem ( F_1 \cup 
  F_2   \cup \{-e_1/2\})  $   with  continuous  boundary values   $ \equiv 0 $ on $ F_2$   and   $ \equiv  | y + e_1/2 |^{2-n} $ on $ F_1. $   
	 We now are in a  position to use a boundary Harnack inequality  proved in  Theorem 3.3 of 
[DS1]   tailored to our situation. 	Let   $   \hat \rho =    \rho -  (1/2) \cot (\xi_2/2),$  and set    $  s = s(y) =    \sqrt{  \hat \rho^2 + y_1^2 }.    $ \,   Next put   \[  W_0 (y)   = \frac{1}{\sqrt{2}}\,
	        \sqrt{ s + \hat \rho } ,   \mbox{ whenever } y \in \rn{n}. \] Note that $ s (y) $ denotes the distance from  $ y $ to $ \ar' F_2,$ and $  \hat \rho $ denotes the signed distance  from $ y' $ to $ \ar' F_2 .$  Here  $ \ar'  F_2 $  denotes the boundary of $ F_2$ relative to $ \rn{n-1}$.  \\

\noi {\bf Theorem A}       Given  $k$  a  positive integer and   $ i = 1, 3,  $  there exists $  a_{i,j} (\hat \rho),  b_j (\hat \rho), $
   $  j  = 1, 2, \dots, $   infinitely  differentiable   whenever 
   \[  | \hat \rho  | <      \hat \la = \min [ (1/10) \cot (\xi_2/2), (1/10) \cot (\xi_1/2)  -  (1/10) \cot (\xi_2/2) ]   \]          with     $ a_{i,0} > 0 $  for $ i = 1, 3, $
     and  for which 	as  $  ( y_1, \rho )  \rar  ( 0, (1/2) \cot (\xi_2/2) )  $    
    \beq   \bea{l}   \label{3.13} (a)  \hs{.2in}   w_1 ( y_1, \rho )   =   W_0 ( y ) \,  {\ds ( \,   \sum_{j=0}^k   a_{1,j} (\hat \rho)  s^j\,  )  +  O ( W_0 (y)  s^{k+1/2} )  }      \\   
	 (b)  \hs{.2in} {\ds  \,  w_3  ( y_1, \rho)   =    W_0 ( y )     ( \sum_{j=0}^k   a_{3,j} (\hat \rho)  s^j )     +  y_{1}  \sum_{0 \leq  j \leq (k-1)/2 }  b_{j}  (\hat \rho)  y_1^{2j} \,  }  \\ \quad  \quad  \,  \, \, +    \, \, \,   O ( W_0 (y) \,  s^{k + 1/2} ) .       \ea  \eeq 
Constants in the  big  O  terms  depend on  $ k, n, \xi_2, \xi_1, $  and  the  $ C^{k + 3/2} $ norms of  $ a_{i,j}  ( \hat \rho) ,  b_j  (\hat \rho), $ \\  for $  i = 1, 3,   1 \leq j \leq k,$ on $ [0, \hat \la].$    We  note that   $ w_1, w_3 $  in   \eqref{3.13}  $ (a), (b) $                                                                                                                                                                       
 have slightly different  expansions  since   $ w_1, $ is even in $ y_1 $ while $ w_3 $   is  not.  
   From   Lemma \ref{lem2.4},  $ \la = d/M, $ \eqref{3.10}, and  \eqref{3.11},    we find   that   as  $ (  y_1, \rho )  \rar  ( 0, (1/2) \cot (\xi_2/2))  $  through coordinates of points not  in  $ F_2, $               \[  \frac{  w_1 ( y_1,  \rho ) }{ w_3 (y_1,  \rho) } \rar  d/M .  \]  which in view of   \eqref{3.13}   $ (a), (b) $  implies  that       
 \beq  \label{3.14}  a_{1,0} (0)  =  (d/M)  a_{3, 0}(0).  \eeq  

Let  $ y = ( y_1, \rho), $  be as in \eqref{3.12}.   From \eqref{3.12} $(a)$  we  find that as $ r \rar 1, $  
\beq  \label{3.15} y_1 \approx 1 - r, \,  \rho = \frac{ r \sin \xi_2}{ (1-r)^2  + 4r \sin^2 (\xi_2/2) }  = O ( y_1^2)  + (1/2) \cot (\xi_2/2) ,  \eeq   where constants  depend only  on  $ \xi_1, \xi_2, n, $  provided $ r \geq 1/2. $  
 From   \eqref{3.11}  $ (f)$,    \eqref{3.13} $ (a), (b), $ \eqref{3.14},   and  $ w_2 ( y ) = |y+e_1/2 |^{2-n} - w_3 (y), $  we find that   if   $ \, \bar  y  =  ( - y_1, \rho), \\  y_0    =  ( y + \bar y )/2 ,  $  then 
       \beq  \bea{l} \label{3.16}  U ( y ) =  M w_1 ( y)  +  d w_2 ( y )  =      
\\  \\ \hs{.1in} (d/2)\,  (   | y  + e_1/2 |^{ 2 - n}   +     | \bar y  + e_1/2 |^{ 2 - n})      +  O ( y_1^{3/2} )  \\ \\ \hs{.1in}  =   U (y_0)   +  O (y_1)^{3/2}  = d   | y_0 + e_1/2 |^{2-n} +  O (y_1)^{3/2}   \mbox{ as } y_1 \rar 0.  \ea  \eeq  
 Finally  we  prove  \eqref{3.8}  and thus finish the proof of  Theorem 1 up to showing the map is 1-1.   Let    $ \ti  U  ( y )  =  | y + e_1/2 |^{n - 2}  U ( y ) . $   From   \eqref{3.12} $(b)$ ,  \eqref{3.11} $(f)$,  \eqref{3.16},     we see that   
\beq  \bea{l}  \label{3.17} 
      \, { \ds  \liminf_{ r \rar 1_-}  \frac{d   -  V (r, \xi_2)  }{ 1 - r} }   = \,      2  \, { \ds  \liminf_{ r \rar 1_-}  \frac{d   -  V (r, \xi_2)  }{ 1 - r^2} }    \\ \\   =  - 
   {\ds   \limsup_{y_1 \rar 0}   \left(  | y + e_1/2|^2     \frac{ [ \ti U ( y )  -  \ti U ( y_0 )]}{2 y_1} \right) }   =  \\ \\ {\ds - (1/2)   |y_0 + e_1/2|^2  U (y_0 )   {\ts \frac{\ar}{\ar y_1}} ( | y + e_1/2|^{n - 2 }  ) (y_0)  =  -  \frac{(n-2) d}{2} } .     \ea \eeq   From \eqref{3.17} we conclude \eqref{3.8}. 

Finally showing the     map $ (\xi_1, \xi_2) \rar (d, M) $ is 1-1 follows easily from the maximum  principle for harmonic functions. Indeed suppose  $ ( \xi_1, \xi_2) ,  (\xi_1', \xi_2') $  are both mapped into  $ (d,M). $  Define $ E_1, E_2 $ and $ E_1',  E_2'$ relative to  $ \xi_1, \xi_2, $ and  $ \xi'_1, \xi'_2, $ respectively. Let $ P, P' $ be the corresponding potentials in Theorem 1.1.  Then either 
$ E_1 \subset E_1'    $   or vice versa and likewise for $ E_2, E_2' $.   If $ E'_1 \subset E_1, $ suppose first that  $ E'_2 \subset E_2  $  Then using the fact that $ P $ is superharmonic  and  $ P' $ is harmonic   in 
$ \rn{n} \sem ( E_1' \cup E_2' ) $  as well as that  both potentials have  the same boundary values,  we obtain from the maximum principle for harmonic functions that  $ P' (0) = 1 < P(0) = 1$ unless $ E_1 = E_1', E_2 = E_2' . $  If  $ E_2 \subset E'_2,$   we  compare  boundary values of 
$ P, P' $  on $ E_1 \cup  E'_2. $   Using   the  maximum principle for harmonic functions 
and \eqref{1.1} $ (c)$  we once again  get   a  contradiction to \eqref{1.1} $(d)$ unless $ P = P'. $        Interchanging the roles of $ P, P'$  we  obtain  $ P = P' $ in all cases. 
 The proof of Theorem 1.1  is now complete.  \end{proof} 
\begin{remark} 
    More sophisticated arguments in \cite{DS2}    yield  that   $  P ( \cdot,  d, M)  $  is   $ C^{1, 1/2} $   in an open neighborhood  of   
$  \{ x \in  \mathbb{S}^{n-1} :  x_1  =  \cos \xi_2  \}.    $ 
 Fix  $ \xi_2 \in  ( 0,  \pi) $  and let     $  \xi_1 $  vary in $ (0, \xi_2).  $   Then from \eqref{2.8}, the maximum principle for harmonic functions,   and  $ \hat \gamma = d/M$ in    Lemma  2.4  we see  that    
\beq  \label{3.18}   d/M  \mbox{ is increasing as a function of  $ \xi_1 $   to say $ d' / M' . $  }  \eeq   Using this fact it is easily seen from Theorem 1 and the maximum principle for \\ harmonic functions   that  $ M $ decreases as a function of $ \xi_1 $ so  $ M' < M . $    Moreover   $ d', M' \rar  1  $  as  $  \xi_1 \rar  \xi_2. $  In  view of our conjecture  and $ \rn{2} $ results  it appears likely that  \beq \label{3.19} \mbox{$ d $ also increases  as a function of  $ \xi_1. $} \eeq    However so far we have not been able to prove this. 
\end{remark} In  $ \rn{2} $  we can prove \eqref{3.19},  without relying on [BL1],    as   follows:  From  \eqref{1.3}, \eqref{1.4},   and  a Schwarz- Christoffel type argument we  have 
\beq  \label{3.20}  \frac{ \ar P ( r, \he,  d, M )}{ \ar \he }  =   \frac{ i  z \,   \ar P (  z ,   d,   M ) }{ \ar z }  = \mbox{ Re } \left[  i  \left(  \frac{1 + z^2 - 2 a z}{ 1 + z^2 - 2 b z } \right)^{1/2}  \right]    \mbox{ for }  z \in \bar  B ( 0, 1) \sem \{ \xi_1, \xi_2 \}. \eeq   where $b = \cos \xi_1,  a  = \cos \xi_2. $   If   $ \xi_1 <  \xi_1'  < \xi_2 = \xi_2' $  and $ P ( z,  d', M' ) $  denotes the extremal potential corresponding to $ \xi_1', \xi_2', $ then  \eqref{3.20}   holds for this potential  with  $ b $ replaced by $ b' = \cos \xi_1'   < b $  .   Moreover    
\beq  \label{3.21}  d' - d  =    \int_0^1   \left[  \left( \frac{1 + r^2 + 2 a r}{ 1 + r^2 + 2 b' r}  \right)^{1/2}   -  
 \left( \frac{1 + r^2 + 2 a r}{ 1 + r^2 + 2 b r } \right)^{1/2} \right] dr/r    > 0.  \eeq     
Also   from   \eqref{3.20}  for $ P (\cdot,d, M),  P(\cdot, d', M'), $ we  see that   if $  \xi_1'  \leq  \he_0  <  \xi_2, $  then   \beq\bea{l}  \label{3.22}   P ( 1, \he_0, d' , M' ) -  P ( 1, \he_0, d, M )   > \\   \\       P ( 1, \he_0, d' , M' ) -  P ( 1, \he_0, d, M )  +  d - d'    =  \\ \\  \hs{.1in}  { \ds     
  \int_{\he_0}^{\xi_2}    \left[  \left( \frac{\cos \he  -  \cos \xi_2 }{\cos \xi_1' -  \cos \he }  \right)^{1/2} \, - \, 
 \left( \frac{\cos \he  - \cos \xi_2 }{ \cos \xi_1  -  \cos \he   } \right)^{1/2} \right]  d\he } > 0 \ea \eeq   Using  \eqref{3.21}, \eqref{3.22},  we  can show that  
\beq  \label{3.23}  \max_{\tau \in [0, \pi] }   \int_0^{\tau} [  P ( 1, \he, d' , M'  ) -    P ( 1, \he, d , M  )  ] d \he = 0  . \eeq 
 Indeed   from  $ M' < M,  \xi_1 < \xi_1', $   the fact that  $ P ( \cdot, d, M )$   is strictly decreasing on  $ ( \xi_1, \xi_2 )  $ and  the first derivative test for maxima,    we  find  first that  the maximum in \eqref{3.23} cannot occur at  some  
 $ \tau  \in (0, \xi_1').$   Moreover  by the same reasoning and  \eqref{3.21},  \eqref{3.22},  this maximum  cannot  occur when $ \tau \in [\xi_1', \pi). $    
 Finally  the integral in \eqref{3.23} $= 0 $ for $ \tau = 0, \pi. $   
From  \eqref{3.23},  \eqref{1.8},  as well as    Baernstein's Theorem (mentioned after \eqref{1.8}),  one can conclude   that  \eqref{1.2}   in $ \rn{2} $  is valid for    
$p = P ( \cdot, d' , M' ). $   

In view of  the above  $\rn{2}$  results,  one  wonders  if in $ \rn{n}, n \geq 3, $  it is true that  for $ \he \in (\xi_1, \xi_2), $ 
\beq  \frac{ \ar^2 P ( 1, \he, d, M )}{\ar \he\,  \ar \xi_1 } < 0 \mbox{ and }  \frac{ \ar^2 P ( 1, \he, d, M )}{\ar \he  \, \ar \xi_2 } >0. \eeq 
 
\begin{remark} 

Another question of interest to us,  is  to what extent  does  Theorem 1.1 generalize  to other PDE' s ?   For example can one replace harmonic in Theorem 1.1 by $ p$ - harmonic  when 
$ 1 < p < n$.    To be more specific, if  $ 1 < p < n, $   does  there exist a  super   $p$-harmonic function   $ u  > 0 $  on  $ \rn{n} $  with  $ u (x) \rar 0 $ as $ x \rar \infty,$  satisfying   
$ (a)-(d) $ of Theorem 1.1  for some  choice of  $ (\xi_1, \xi_2), 0  \leq  \xi_1 < \xi_2 \leq \pi,    $  and 
   $ (d, M),  0  < d < 1 < M \leq \infty $  (see \cite{HKM} for relevant definitions).  
\end{remark} 

\section{Proof of Proposition 4.1 }\label{sec4} 
\label{sec4}  
\setcounter{equation}{0} 
 \setcounter{theorem}{0}
  Recall  from  \eqref{2.5}, \eqref{2.6}   that  \beq  \bea{l} \label{4.1}  h ( r, \he, \he_1)  =    {\ds \int_0^{\pi}   \left( 1 + r^2 - 2r  \psi (\he, \he_1, \ph) \right) ^{1-n/2}   (\sin \ph)^{n-3}  \,  d \ph }  
\mbox{ with } \\ \\   
\psi (\he, \he_1, \ph)  =  \cos \he \cos \he_1 + \cos \ph \sin \he \sin \he_1 . \ea \eeq Theorem 1.2  is an easy consequence of the following  proposition.  
\begin{proposition} \label{prop4.1}  
 $ \frac{ \ar^2  h (1, \he, \he_1) }{ \ar \he \, \ar \he_1}  < 0 $  whenever $ \he_1  \not = \he, \he, \he_1  \in [0,\pi]. $  \end{proposition} 
\begin{proof}

Since $ h(1, \he, \he_1) $  is symmetric in  $ \he, \he_1, $   we assume as we may that \\ $ 0 \leq \he_1 < \he \leq \pi.   $   
 Differentiating \eqref{4.1} we obtain    
\beq  \bea{l}  \label{4.2} 
  {\ds \frac{ \ar h ( 1, \he, \he_1 )}{ \ar \he}  =  (n - 2 )  2^{-n/2}  \int_0^{\pi}  \frac{\ar \psi}{\ar \he}    
( 1 - \psi )^{-n/2} ( \sin \ph )^{n-3} d \phi}  \\ \\ =  {\ds  (n - 2 )  2^{- n/2 }  \int_0^{\pi}  [   \cos \phi \cos \he  \sin \he_1  - \sin \he \cos \he_1 ]  f ( \he, \he_1, \ph ) d \ph \ds }  
 \ea  \eeq  
 where \[  f ( \he, \he_1, \ph ) =         
   ( 1 - \psi (\he, \he_1, \ph ) )^{-n/2} ( \sin \ph )^{n-3} \,  .  \]   From\eqref{4.1} we  see that \eqref{4.2} can be rewritten as    
\beq  \bea{l}  \label{4.3} {\ds \frac{ \ar h ( 1, \he, \he_1 )}{ \ar \he}   =   - (n/2 -1)  \cot \he \, h (1, \he,  \he_1 )  +      
  \frac{ (n-2) 2^{-n/2}  (  \cos \he - \cos \he_1 )}{ \sin \he} \int_0^\pi  f  ( \he, \he_1,  \ph) d \ph }. 
 \ea \eeq    Taking second partials in \eqref{4.3} we have 
\beq \bea{l} \label{4.4} 
{\ds \sin \he  \frac{ \ar^2 h ( 1, \he, \he_1 )}{ \ar \he \, \ar \he_1}   =   - (n/2 -1)  \cos \he \,  \frac{ \ar h (1, \he,  \he_1 ) }{ \ar \he_1}  +      
  (n-2) 2^{-n/2}   \sin \he_1  \int_0^\pi  f  ( \he, \he_1,  \ph) d \ph } 
\\ \\  \hs{.874in} +  (n-2) 2^{-n/2}  { \ds (  \cos \he - \cos \he_1 ) \int_0^\pi  \frac{ \ar f  ( \he, \he_1,  \ph)}{ \ar \he_1}   d \ph } =  J_1  + J_2  + J_3 \ea \eeq  Since $ h $ is symmetric in $ \he, \he_1 $ we can interchange $ \he, \he_1 $ in 
\eqref{4.2}, \eqref{4.3},  to get 
\beq  \label{4.5}  J_1 =    ( n/2 - 1 )^2  \cos \he  \cot \he_1   h ( 1, \he, \he_1)      
 -   \frac{ (n-2)^2  2^{-(n+2)/2}  \cos \he  (  \cos \he_1 - \cos \he )}{ \sin \he_1} \int_0^\pi  f  ( \he, \he_1,  \ph) d \ph .  \eeq  Moreover,  as in  \eqref{4.3}  
\beq  \bea{l}  \label{4.6}   J_3 = - (n/2) (n-2) 2^{-n/2}  ( \cos \he - \cos \he_1 )  \cot \he_1   { \ds \int_0^{\pi}   f ( \he, \he_1, \ph) d \ph} \\ \\     - ( n - 2)  (n/2)  2^{-n/2} { \ds \frac{( \cos \he_1 - \cos \he )^2 }{ \sin \he_1 } \int_0^\pi ( 1 - \psi )^{-1}  f ( \he, \he_1, 
\ph ) d \ph }  =  J_4 +   J_5 \ea  \eeq   Adding $ J_2 $  in \eqref{4.4} to  $ J_4 +  J_5 $ in \eqref{4.6}  we get   
\beq \bea{l} \label{4.7}  J_2 + J_4 + J_5  =    [ (n-2) 2^{-n/2} \sin \he_1  - (n/2) (n-2) 2^{-n/2}  ( \cos \he - \cos \he_1 )  \cot \he_1  ] { \ds \int_0^{\pi}   f ( \he, \he_1, \ph) d \ph} \\ \\ 
- ( n - 2)   (n/2)  2^{-n/2} { \ds \frac{( \cos \he_1 - \cos \he )^2 }{ \sin \he_1 } \int_0^\pi ( 1 - \psi )^{-1}  f ( \he, \he_1, 
\ph ) d \ph } \ea \eeq  Finally we arrive at  
\beq  \bea{l}  \label{4.8} (n/2 - 1)^{-1}    2^{n/2 - 1} (J_1 + J_2 +  J_4 + J_5)   =    ( n/2 - 1 )  \cos \he  \cot \he_1 \, {\ds  \int_0^{\pi} ( 1 - \psi ) f ( \he, \he_1, 
\ph )  d \ph }  \\ \\ \hs{.2 in}{ \ds   + \, \left[  \, (n/2 - 1 ) \frac{  ( \cos \he_1 - \cos \he )^2}{ \sin \he_1}   +  \,  \frac{  1 - \cos \he \cos \he_1 }{ \sin \he_1 } \right]
    \int_0^\pi  f ( \he, \he_1, \ph ) d \ph  }   \\ \\  \hs{.2 in}-   (n/2)  { \ds \frac{( \cos \he_1 - \cos \he )^2 }{ \sin \he_1 }  \int_0^{\pi} ( 1 - \psi )^{-1} f ( \he, \he_1, \ph )}  d \ph     \ea   \eeq   

From \eqref{4.4} - \eqref{4.8}  we conclude   that to finish the proof of Proposition 4.1  we need to show that the right hand side of \eqref{4.8}  is negative.  To   do this  we let $ a = \sin \he \sin \he_1,  b = 1 - \cos \he \cos \he_1 $  and use the beta function  to  calculate in terms of our previous notation: 
\beq  \label{4.9}   A = b^{n/2 - 1}\int_0^{\pi}  ( 1 - \psi )^{ 1 - n/2} (\sin \ph )^{n-3} d \ph  
 =   \left( \sum_{l=0}^{\infty}  \frac{ (n/2 - 1)_{2l} \,  \Ga (n/2 - 1 ) \,   \Ga ( l + 1/2) }{ (2l)! \, \, \Ga (  n/2 +  l - 1/2) }  \left(\frac{a}{b}\right)^{2l}  \right)     \eeq  
where $ \Ga $ is the Gamma function and  $ (\la)_k  = \la ( \la + 1) \dots ( \la + k - 1 ) $ when $ k $ is a positive integer with $ (\la)_0 = 1 =  (0){\ts !}  \, . $   
\beq \label{4.10}   B = b^{n/2 }   \int_0^{\pi}  ( 1 - \psi )^{ - n/2} (\sin \ph )^{n-3} d \ph  
 =  \left( \sum_{l=0}^{\infty}  \frac{ (n/2 )_{2l} \,  \Ga (n/2 - 1 ) \,   \Ga ( l + 1/2) }{ (2l)! \, \, \Ga (  n/2 +  l - 1/2) }  \left(\frac{a}{b}\right)^{2l} \right)   \eeq  
\beq  \label{4.11}   C =  b^{n/2 +1} \int_0^{\pi}  ( 1 - \psi )^{ - (n/2 + 1) } (\sin \ph )^{n-3} d \ph  
 =  \left( \sum_{l=0}^{\infty}  \frac{ (n/2  + 1 )_{2l} \,  \Ga (n/2 - 1 ) \,   \Ga ( l + 1/2) }{ (2l)! \, \, \Ga (  n/2 +  l - 1/2) }  \left(\frac{a}{b}\right)^{2l} \right)   \eeq  
In terms of this notation, $ 0 \leq  a  \leq 1,  a < b   \leq  2,   $   and  \eqref{4.4} - \eqref{4.11},  we get at $ (1, \he, \he_1 ), $   \beq  \bea{l} \label{4.12} (n/2 - 1)^{-1} (2b)^{n/2-1} \sin \he_1 \sin \he { \ds  \frac{ \ar^2  h (1, \he, \he_1) }{ \ar \he \, \ar \he_1} } \\ \\ = 
 (n/2 - 1)^{-1}  2^{n/2 - 1} b^{n/2-1}  (\sin \he_1)  ( J_1 + J_2 +  J_4 + J_5 )    
\\ \\ =  ( n/2 - 1 )  ( 1 -  b )    A    { \ds   + \,  b^{-1} \left[ (n/2  - 1)    \, 
   (b^2  - a^2)   +  \,  b  \right]
    \, B}    \\ \\   - b^{-2}     (n/2)   { \ds ( b^2  - a^2 )    C}  \\ \\  =  b (n/2 - 1 ) (B - A)   +  (n/2 - 1) (A-C)  +   ( B - C )   \\ \\    - (n/2 - 1 ) (a^2/b) \, B    +  (n/2)  (a^2/b^2)  C     ={\ds  \frac{\Ga (n/2 - 1)  
 \Ga(1/2) }{\Ga ( n/2 - 1/2 )} }  D  \ea \eeq

if $ n = 2m, $ the $ l = 0 $ term   in the sum for $ D $  is     
\beq  \label{4.13}   T_0^2 = -   (m - 1 ) (a^2/b )  +   m (a^2/b^2) .       \eeq  
Let $ A_l,  B_l,  C_l, l = 1, 2, \dots  $  be  the  $ l $ th nonzero coefficient multiplying  $  (a/b)^{2l} $ in  $ A. 
B, C. $  If  again $ n = 2m, $ then 
 \beq \bea{l} \label{4.14}   {\ds    (B  - A)_1  = ( \frac{m+1}{m-1} - 1 ) A_1 = \frac{2}{m-1} A_1 ,  \, \,   \,    (C- A)_1  = ( \frac{(m+1)(m+2)}{m(m-1)} - 1 ) A_1 =   \frac{ 4m + 2}{(m-1)m} A_1,}  \\ \\ 
                      {\ds ( C -  B)_1  =    \frac{ 2(m + 1 )}{(m-1) m }  A_1.  }  \ea \eeq   Also 
\beq   \label{4.15} \,\frac{\Ga (m - 1 /2) }{\Ga(1/2) \Ga (m-1)}  A_1 =   
\frac{ (m-1) m   }{2  (2 m-1) } .  \eeq 
 Using \eqref{4.14}- \eqref{4.15}   we deduce that the part of the sum in \eqref{4.12}   involving $ A_1, B_1, C_1,  $ in  $D$ is \beq \bea{l} \label{4.16}  
{\ds \left[ 2 b   - 4  - (2/m)   -   2 \frac{m +1}{( m - 1 ) m }  -  (m+1) a^2/b    +      \frac{  (m +1)(m+2) }{ m - 1}   a^2/b^2 \right] \frac{ (m-1)\,  m \, (a/b)^2  }{ 2(2 m-1) } }  =     
	\\ \\   \left(  {\ds  \frac{ m( m - 1)}{  2 m - 1}  a^2/b  - \frac{4m^2}{ 4m-2} a^2/b^2 } \right)   + {\ds  \left( -  \frac{  (m-1) m (m+1)}{4m - 2 }   a^4/b^3  +       \frac{m (m+1)(m + 2 )}{4m - 2}   a^4/b^4 \right)}  \\ \\ \hs{.83in} {\ds =   T_1^1 \hs{1.25in}  + \hs{ 1.4in} T_1^2 }   \ea \eeq  
  From   \eqref{4.13}, \eqref{4.16},  we get  
 \beq \bea{l} {\ds  \label{4.17} T_0^2 +   T^1_{1} =   - \frac{ (m -1)^2}{( 2m - 1 )}  a^2/b   \,-   \frac{ 2m}{4 m-2 }\,  a^2/b^2 } < 0 \ea \eeq   
Fix $ l $ a positive integer and once again set  $ n = 2m.$  As in the case $ l = 1 $ we see  from    \eqref{4.9} - \eqref{4.11} that 
\beq \bea{l} \label{4.18}  {\ds (B - A)_l =  ( \frac{m + 2l - 1 }{m - 1} - 1 ) A_l  =  \frac{2l}{m-1} A_l }, \\ \\  (C-A)_l  = {\ds [\frac{ ( m + 2l - 1)(m+2l)}{(m-1)m} - 1]A_l  = \frac{4lm + 2l (2l-1)}{(m-1)m}}A_l  \\  \\   ( C - B )_l  = {\ds \left[ \frac{ ( m + 2l - 1)(m+2l)}{(m-1)m} -  \frac{m + 2l - 1 }{m - 1}  \right]  A_l =        \frac{ 2l ( m + 2l - 1)}{(m-1)m} A_l }  \ea \eeq  
  Also  \beq    \label{4.19}        \frac{\Ga (m - 1 /2) }{\Ga(1/2) \Ga (m-1)}   A_l =   \frac{ (m-1)_{2l}  (l-1/2) (l - 3/2) \dots (1/2)}{ (2l) ! (m+l - 3/2) ( m + l - 5/2 ) \dots  (m-1/2)} 
\, .  \eeq   
Using \eqref{4.18}, \eqref{4.12},  we see that the terms in   $ D $  involving   $ (a/b)^{2l}  $ times  $ A_l, B_l, C_l $  are ,  
\beq  \bea{l} \label{4.20}  \left[ [2l b     -    4l - (2l/m) (2l - 1 ) - \frac{2l (m+2l - 1)}{m (m-1) } \right] A_l  \,  (a/b)^{2l}   +  \\ \\  
  {\left[  -{\ds  (a^2/b) (m + 2l - 1 )   +  \frac{ (m + 2l - 1)(m+2l)}{m-1}   (a^2/b^2) } \right]   A_l }\,  (a/b)^{2l}   = T^1_{l} + T^2_{l}   \ea \eeq                                                                        
We claim that   \beq  \label{4.21} \sum_{k=0}^l  T^1_{k} + \sum_{k=0}^{l-1} T^2_{k}  \leq 0   \mbox{ for } l = 1, 2, \dots   \eeq 
where $ T^1_{0} = 0. $    Note from  \eqref{4.17} that \eqref{4.21} is true when  $ l = 1.$  Proceeding by induction assume  \eqref{4.21} is  true  for some  positive integer $ l $.   We  need to show that  
\beq  \bea{l}  \label{4.22} \left[ [2(l+1) b     -    4 (l+1) - {\ds \frac{2 (l+1)(2l+1)}{m} - \frac{2(l+1) (m+2l +1)}{m (m-1) }} \right] A_{l+1} \,  (a/b)^2 \\ \\ 
 +     \left[  -{\ds  (m + 2l - 1 ) \frac{a^2}{b}    +  \frac{ (m + 2l - 1)(m+2l)}{m-1}   (a/b)^2 } \right]   A_l   < 0        \ea    \eeq                
  Now   \beq   \label{4.23}  A_{l+1}/ A_l   =   \frac{ (m + 2l)(m + 2l - 1 )(l+1/2) }{ (2l + 2 )( 2l + 1)(m+l - 1/2) } = \frac{ (m+2l)(m+2l - 1)}{4 (l+1)(m + l - 1/2)}   \eeq  From \eqref{4.23} we get for the $ a^2/b $ term in  \eqref{4.22} 
    \beq  \label{4.24}   2(l + 1)  b \, A_{l+1}\,  (a/b)^2 -{\ds (m + 2l - 1 )  (a^2/b)   } A_l  = - \frac{ (m - 1)( m + 2l - 1 ) }{ 2 ( m + l - 1/2)}\,  (a^2/b) \, A_l <  0. \eeq  
Next we observe  from \eqref{4.23}  for the $ A_{l+1} (a/b)^2 $ term in \eqref{4.22} that     \beq  \bea{l} \label{4.25}    {\ds  \left( -    4 (l+1) - \frac{2 (l+1)(2l+1)}{m} - \frac{2(l+1) (m+2l +1)}{m (m-1) } \right) A_{l+1} (a/b)^2  }\\ \\ = 
- {\ds \frac{4 (l+1)(m+l) }{m-1} A_{l+1} (a/b)^2  =   - \frac{ (m+l)(m+2l)(m + 2l - 1) }{ (m-1)(m+l-1/2)} A_l  (a/b)^2  }        \ea \eeq      Hence  adding   the $ (a/b)^2 A_l $ term in \eqref{4.22}  to the right hand side of \eqref{4.25},  
\beq   \bea{l} \label{4.26}  {\ds    - \frac{ (m+l)(m+2l)(m + 2l - 1) }{(m-1)(m+l-1/2)} A_l  (a/b)^2   +  \frac{ (m + 2l - 1)(m+2l)}{m-1} A_l \, (a/b)^2 } \\ \\ 
 =  {\ds - \frac{ (m + 2l - 1)(m+2l)}{2(m-1)(m+l - 1/2)  } A_l (a/b)^2  < 0 . }  \ea \eeq 
 Finally adding  right hand  sides of  \eqref{4.26},  \eqref{4.24},  we find from \eqref{4.22}  and induction  that   claim \eqref{4.21} is true.    From the definition of  $ A, B, C $  we see that these functions have absolutely convergent  series involving powers of $ (a/b)^{2l} $ and thereupon  that   \eqref{4.21}  converges to  $  \frac{\Ga (n/2 - 1)  
 \Ga(1/2) }{\Ga ( n/2 - 1/2 )}  D. $   So  $ D < 0 $ and it follows from \eqref{4.12}  that  Proposition 4.1  is valid  when $ n \geq 3 $ is a positive integer.   
\end{proof} 
  \section{Proof of Theorem 1.2 }\label{sec5} 
\setcounter{equation}{0} 
 \setcounter{theorem}{0} \begin{proof} 
Recall from section 1  that if $ \xi_2 = \pi, 0  < \xi_1 < \pi, $ and $ P =  P ( \cdot, d, M ) $ 
is the corresponding extremal potential satisfying  $(a) - (d) $  of Theorem 1.1,  then $ P $ is harmonic in $ \rn{n} \sem E_1, $ so 
\beq \label{5.1}  d   =  P ( -  e_1, d, M ),  \mbox{  and  \eqref{1.2} holds whenever }  p \in  
\mathcal{F}_d^M. \eeq  
Thus to prove Theorem  1.2 we show for $ d $  as in \eqref{5.1} and $ 1 < M < \infty,$ 
that  \beq \label{5.2} \int_{\mathbb{S}^{n-1}} \Phi (P( r y, d, M )) \, d \mathcal{H}^{n-1} \leq       
 \int_{\mathbb{S}^{n-1}} \Phi (P( r y, d, \infty)  ) \, d \mathcal{H}^{n-1},  0 < r < \infty. \eeq 
 To do so we first note for any $  p \in \mathcal{F} $  that  $  2^{2-n} \leq  p  $  in $ B (0, 1 ) $ as follows from the minimum principle for potentials and the fact that any two points on $ \mathbb{S}^{n-1} $ are at most distance two apart. Also in [L] we showed the existence of $ P ( \cdot, d, \infty ) $  in\\ Theorem \ref{1.1} whenever  $ 2^{2-n} \leq d < 1. $  Thus $ P ( \cdot, d, \infty) $ exists when  $ d $ is as in \eqref{5.1}.  

The proof of  \eqref{5.2} is by contradiction.    
For ease of writing  we put 
  $ P = P ( \cdot, d, M ),$ $  P' = P ( \cdot, d, \infty ), $ and  write  in spherical coordinates,  $ P ( r, \he),  P' ( r, \he ), $ which is permissible,  since  both functions are   symmetric about the  $ x_1 $  axis.   Also for fixed   $ \xi_1 \in  (0, \pi ) , $ let $ E_1, $  be as defined  in Theorem 1.1 relative to $ P $  while 
$ \xi'_2  \in (0, \pi),    E'_2,   $ are defined relative to  $P'.$ 
As in \eqref{2.3}-\eqref{2.6},  and from Theorem 1.1  we deduce the existence of  positive Borel measures, $ \nu, \mu, $  on $ [0, \pi ] $  with total mass $ c_n^{-1}$  corresponding to  $ P, P' , $ respectively.  
   $ \nu $ has its support  in $ [0, \xi_1] $  while $ \mu $ has its support in $\{0\} \cup [\xi_2', \pi].$ 
Moreover    \beq  \bea{l} \label{5.3} (a) \hs{.2in} P ( 1, \he ) = c_n { \ds  \int_0^{\xi_1} h ( 1, \he, \he_1 ) \, d \nu (\he_1 ) } \mbox{ and }  \\ \\  (b) \hs{.2in} P' ( 1, \he ) =  c_n \,  \al \, h ( 1, \he, 0  )  + c_n  { \ds \int_{\xi'_2}^{\pi} h ( 1, \he, \he_1 ) \, d \mu ( \he_1 ) } \mbox{ with } 
  \al = \mu ( \{e_1\} ) > 0.  \ea  \eeq     
 We note  from Theorem \ref{1.1}   that $ P (r, \cdot ),  P' ( r,  \cdot ),  $  are  continuous in the extended sense  and non increasing  on $ [0, \pi] $  whenever  $ 0 < r < \infty. $    From this fact,  \eqref{1.7}, 
\eqref{1.8},  and   \\ d$\mathcal{H}^{n-1} = c_n (\sin \he)^{n-2} d\he,  $  we deduce that  to prove  \eqref{5.2} it suffices to show 
\beq \label{5.4}   \int_0^{\tau}  P ( r, \he ) \sin^{n-2} \he d\he  \leq  \int_0^{\tau}  P' ( r, \he ) \sin^{n-2} \he d\he \eeq whenever $ 0 <  r < \infty, 0  \leq \tau \leq \pi$.     Moreover from the Baernstein maximum principle (see  the discussion after \eqref{1.7}) ,  we need only  prove  \eqref{5.4}   when $ r = 1. $   To do this observe that if  \eqref{5.4}    is false for some $  \hat \he \in (0, \pi), $ when $ r = 1, $   then there exists $ \bar \he \in (0, \pi ) $   with      
\beq \label{5.5}  0 <  \max_{\tau \in [0, \pi ] }   \int_0^\tau  ( P - P')  (1, \he ) \sin^{n-2} \he d \he  =  \int_0^{\bar \he }   ( P - P' ) (1, \he )   \sin^{n-2} \he d \he.  \eeq     Since  $ P'  $ is  strictly decreasing on $ [0, \xi'_2]  $  with $ P' \equiv d  <  P $ on $ [\xi_2', \pi),$ $ P \equiv M $  on $ [0,  \xi_1 ], $   and the integrals in \eqref{5.4}  are equal when $ \tau = \pi, $    we see from the first derivative test in calculus, 
    that  $ \xi_1 < \xi'_2$  and we  may assume 
\beq \label{5.6}   \bar \he \in [\xi_1   ,  \xi_2' ],  \mbox{  with }  P ( 1, \bar \he ) = P' ( 1, \bar \he ). \eeq   To get  a  contradiction  we  note  from Proposition 4.1  and  
\eqref{5.3} $(a)$ that if  $ \he \in (\xi_1, \xi_2'), $     
  \beq  \label{5.7} -  \frac{\ar P}{\ar \he  } ( 1,  \he)   = -  c_n  \int_0^{\xi_1}  \frac{ \ar h}{\ar  \he }  ( 1, \he, \he_1 ) d \mu (\he_1)   \geq  - c_n \, \frac{ \ar h}{ \ar \he }  ( 1, \he, 0 ).  \eeq 
    On the other hand   from Proposition 4.1 we have   for  $  \he, \he_1  \in  [0, \pi ],  \he_1 > \he,  $     
\beq \label{5.8} \frac{ \ar h}{\ar \he_1} ( 1, \he, \he_1 )    >    \frac{ \ar h}{\ar \he_1} ( 1, \he, \pi   )   = (n/2 - 1) \sin \he ( 1 + \cos \he )^{-n/2} > 0. \eeq
Using \eqref{5.8} amd \eqref{5.3}$ (b) $ we see  for  $ \he \in (\xi_1, \xi_2'), $  that 
  \beq   \label{5.9} \frac{\ar P'}{\ar \he  } ( 1,  \he)    >      c_n \, \al \, \frac{ \ar h}{\ar  \he }  ( 1, \he, 0  )  =  -  c_n \al   (n/2 - 1) \sin \he ( 1 - \cos \he )^{-n/2}.   \eeq    Combining   \eqref{5.7} and \eqref{5.9}  we get   
\beq \label{5.10}   \frac{\ar P'}{\ar \he  } ( 1,  \he)   -  \frac{\ar P}{\ar \he  } ( 1,  \he)      > c_n  ( \al - 1 )  \frac{\ar h}{ \ar \he}  (1, \he, 0 )  > 0 \mbox{ on } ( \xi_1, \xi_2'),  \eeq   
since $ 0 < \al < 1. $  Thus $ P' - P $  is increasing on  $  ( \xi_1, \xi_2' ], $  so
\beq   \label{5.11} P '( 1,  \he ) - P ( 1,  \he)     <      P' ( 1, \xi_2 ) - P ( 1, \xi_2 ) =   d - P ( 1, \xi_2 )  < 0,   \eeq   since  $ P $ is strictly decreasing on $ [\xi_1, \pi]  $ 
with $ P ( 1, \pi ) = d. $  \,   Letting  $ \he \rar \bar \he $     we arrive at a  contradiction to    \eqref{5.5}.     From this contradiction we conclude Theorem 1.2.   \end{proof}     
\begin{remark} 
Conjecture 1  when  $ E_2 \not = \{ - e_1  \} $  or even \eqref{1.9} when $ E_1 \not = \{e_1\}, $   seems   difficult.
  If $ E_2 \not =  \{-e_1\}, $   the main problem is  that the proposed extremal potential, $ P ( \cdot, d, M ) $  must have  mass   at points on  $  \mathbb{S}^{n-1} $   where it assumes   both  its  maximum  and   minimum values on 
 $  \bar B (0, 1 ). $   This splitting of  the mass  seems  to  rule out   an immediate proof of  Conjecture 1     using     Baernstein's   $ * $  function.  A simpler problem  
in  view  of   \eqref{1.8} is to show  for  $ p \in \mathcal{F}_d^M :$  
    \\  
\beq \label{5.12}     \int_{  \{ y \in  \mathbb{S}^{n-1}: y_1 \geq \cos \he_0 \}}  p (r y )   \, d \mathcal{H}^{n-1} y   \, \leq \, 
\int_{ \{ y \in  \mathbb{S}^{n-1}: y_1 \geq \cos \he_0  \} }  P ( ry,  d, M )  
 \mathcal{H}^{n-1} y  \eeq whenever  $ 0 \leq \he_0 \leq \pi,  0 < r < \infty, $   
  Note   that  a positive  answer to  \eqref{5.12} would imply  \eqref{1.9}    when  $ E_1 \not = 
\{ e_1\} $ and also \eqref{3.19}.  
\end{remark}    To    indicate  our efforts in trying  to prove \eqref{5.12},  we    first observe  that it  suffices  to  prove  \eqref{5.12} when $ p $ is symmetric 
about the $ x_1 $  axis, so $ p (x)  = p ( r, \he ) $  when $ |x| = r, x_1 = r \cos \he. $     Also,  using the  Baernstein * function, as mentioned earlier,  we need only prove  \eqref{5.12}  when $ r = 1. $   Next as in  \eqref{2.4},  we see that   
\beq \label{5.13} p ( 1, \he) = c_n  \int_0^{\pi}   h ( 1, \he, \he_1 ) d \si ( \he_1 ) \eeq  
where $ h$ is as in \eqref{2.5}  and  $ \si $ is a positive Borel measure on $ [0, \pi] $  with  $ \si ([0, \pi] ) = c_n^{-1}. $ 
 Given $ \he_0 \in (0, \pi), $  it follows from the  Fubini theorem   and  \eqref{5.13}  that  
\beq \label{5.14}  \int_{ 0}^{\he_0 }  p (1, \he)  \sin^{n-2} \he d \he =   \int_0^{\pi}  k ( 1, \he_1 ) 
d \si (\he_1 )  \eeq    where 
\beq \label{5.15} k (1, \he_1)  =  c_n  \int_0^{\he_0}  h ( r, \he,  \he_1 ) \sin^{n-2} \he  d \he \eeq
We note  that  $  k ( 1, \cdot) $ is continuous  on  $ [0, \pi] $  (in fact H\"{o}lder continuous)  so from  a theorem on weak convergence of measures, we deduce that 
\beq  \label{5.16}  \sup_{p \in \mathcal{F}_d^M}  \int_{ 0}^{\he_0 }  p (1, \he)  \sin^{n-2} \he d \he
=   \int_{ 0}^{\he_0 }  \ti P  (1, \he)  \sin^{n-2} \he d \he \eeq for  some 
$  \ti P \in \mathcal{F}_d^M .  $ 
As for  $ \ti P, $  we can prove 
when  $ M =  \infty :$  \begin{lemma} \label{lem5.2}   Let  $ \ti \nu  $  be the positive Borel measure on $ [0,\pi] $   corresponding  to $ \ti P $  in \eqref{5.16}.  
 Then 
  \beq  \, \label{5.17}  \, \ti \nu \{ \he_1 :   \he_0 < \he_1  \leq \pi \mbox{ and } \ti P  (1,  \he_1 ) \not = d \} = 0 \eeq 

\end{lemma}

\begin{proof} 
  The  proof of  \eqref{5.17} is by contradiction  and  essentially given in  Lemma 4   of  \cite{L}.     To briefly outline the argument,  if \eqref{5.17} is false, then  using lower semi-continuity of $ \ti P, $   one  can  assume there exists  $ \tau_i, 1 \leq i \leq 3, $ with $ \he_0 < \tau_1 <   \tau_2 < \tau_3 \leq \pi,  $ $ \ti P  ( 1, \cdot )  >  d,  $   on  $[\tau_1, \tau_3],$ and $ \ti\nu ([\tau_1, \tau_2]) = \ti   \nu ( [\tau_2, \tau_3] ). $  For  $ \ep_0 $ small, $ \ep \in [- \ep_0, \ep_0 ]\sem \{0\}, $  and  $  \tau_1 < \he_1 < \tau_3 $  define  $ \ti g (\he_1, \ep ) $ by 
    \beq   \bea{l}  \label{5.18}  (a) \hs{.2in}  0 < \ti g ( \he_1, \ep) < \pi \\ (b) \hs{.2in} (1+ \ep) [ ( 1 + \cos \ti g (\he_1, \ep) )^{1-2/n} + B ] =  
  ( 1 + \cos \he_1 )^{1-2/n}  + B  \ea \eeq   where  $ B > 0 $ is to be chosen.  One can verify  from 
\eqref{5.18}   that $  \ti g ( \he_1, \ep )  $  is an increasing function of  $ \he_1 $  on $  [\tau_1, \tau_3].  $  For   $ \ep \in (0, \ep_0], $  let  
  \beq \bea{l} \label{5.19}    \ti P_\ep ( r, \he )  = c_n  ( 1 + \ep )  { \ds \int_{\tau_1}^{\tau_2}  h ( r, \he ,  \ti g (\he_1, \ep ) ) d\nu ( \he_1) }
  \\  \hs{.5in} +  c_n ( 1 - \ep )  { \ds \int_{\tau_2}^{\tau_3}  h ( r, \he ,  \ti g (\he_1, - \ep ) ) d\nu ( \he_1) }  \\   
	    \hs{.5in} +  c_n    { \ds \int_{[0,\pi] \sem [\tau_1,\tau_3]}  h ( r, \he ,  \he_1  )  d\nu ( \he_1) }.  
  \ea \eeq  Next   since    $ \ti g ( \he_1, \pm  \ep )  $  is an increasing function of  $ \he_1 $  on $  [\tau_1, \tau_3]  $    it follows that  
  $ \ti P_{\ep} $ is a potential symmetric about the $ x_1 $ axis.       
     A lengthy  calculation     then   
    gives  for $ B $ sufficiently large that
  $\ti P_{\ep} \in  \mathcal{F}_d^{\infty} $  and  
  \beq  \label{5.20}   \frac{\ar \ti P_{\ep} (r, \he )}{ \ar \ep} > 0   \mbox{ when  $ 0 < \ep < \ep_0, 0 < r < \infty, $ and 
  $ \he \in [0, \he_0].$ } \eeq  
  From \eqref{5.20}   we obtain    \beq  \label{5.21}  \int_0^{\he_0}  \ti P ( r, \he) ( \sin \he )^{n-2} d \he   <     \int_0^{\he_0}  \ti P_{\ep}  ( r, \he) ( \sin \he )^{n-2} d \he \eeq   which  contradicts    \eqref{5.16}   and  so Lemma  1 is true.  \end{proof} 

\begin{remark}    Using   \eqref{5.17}  of   Lemma 5.2,   one can  get  as in 
[L]  that  for some\\ $ \hat \he_0  \in [\he_0, \pi], $   
   \beq  \label{5.22}  \ti P ( 1, \he )  =  d  \mbox{ on }   [\hat \he_0, \pi].  \eeq       To use  the  above  `mass moving '  argument   further  in  proving 
\eqref{5.12}    appears     difficult 
  since     from the mean value  property for harmonic functions,  \[ \int_0^{\pi}    \frac{\ar \ti P_{\ep} (1, \he )}{ \ar \ep}  (\sin \he )^{n-2} d \he   =  c_n^{-1}   \frac{\ar \ti P_{\ep} (0, 0 )}{ \ar \ep} = 0 . \] 
     Thus   moving mass  inside  $ (0, \he_0), $ as above,   would  create points  in this interval  where $ \ti P_{\ep} < \ti P  $   and  so  perhaps  not imply  \eqref{5.20}.    
 
The lengthy   calculation  mentioned above was  to show    that 
  \beq  \bea{l} \label{5.23} 
\frac{\ar}{\ar \he_1}   \left[ \frac{ \frac{\ar h (1, \he, \he_1 ) }{\ar \he_1 } }{ \frac{\ar h (1, \pi,  \he_1 ) }{\ar \he_1 }}  \right]   > 0   \ea \eeq     
  whenever  $   \he,  \he_1, \in  (0, \pi)$  and $ \he \not = \he_1. $   Originally in
\cite{L},  after many months of  trying,  I  had just proved  \eqref{5.23}  for $ n = 3$      so  \eqref{1.9} was just valid  in $ \rn{3}.$  Still  I  submitted my paper to the Proc. of LMS and after a  few months  received a  handwritten report from the referee  (Walter Hayman) to the  effect   that  I  should try using the substitution,      \[   \frac{\sin \he  \sin  \he_1 }{ 1 -  \cos \he  \cos \he_1} =    \frac{2 t }{ 1 + t^2 }  \mbox{ with }   t   =   \frac{ \tan (\he/2)}{ \tan (\he_1/2)} \]  to  simplify my  calculations.   
 Using this observation    I  was eventually able to  prove \eqref{5.23} and after that use the above contradiction argument to  get 
 \eqref{1.9} in $ \rn{n},  n \geq 3.  $    
Note  that  Proposition 4.1  is in the same spirit as  \eqref{5.23}. 
\end{remark}

\end{document}